\documentclass[12pt]{amsart}
\usepackage[latin1]{inputenc}
\usepackage{mathptmx}
\usepackage{amscd}
\usepackage{amssymb}
\newtheorem{theorem}{Theorem}

\newtheorem{corollary}[theorem]{Corollary}

\theoremstyle{definition}
\newtheorem{remark}[theorem]{Remark}
\newtheorem{definition}[theorem]{Definition}

\let\phi=\varphi

\def\Ann{\operatorname{Ann}}
\def\Ass{\operatorname{Ass}}

\let\oldbigwedge\bigwedge
\def\BIGwedge{{\textstyle\oldbigwedge}}
\def\medwedge{{\scriptstyle\oldbigwedge}}
\def\bigwedge{\mathchoice{\BIGwedge}{\BIGwedge}{\medwedge}{}}

\let\epsilon=\varepsilon

\begin{document}
\title{Zero-divisors of Semigroup Modules}

\author{Peyman Nasehpour}
\address{Universit\"at Osnabr\"uck, FB Mathematik/Informatik, 49069
Osnabr\"uck, Germany} \email{Peyman.Nasehpour@mathematik.uni-osnabrueck.de} \email{nasehpour@gmail.com}

\begin{abstract}
Let $M$ be an $R$-module and $S$ a semigroup. Our goal is to discuss zero-divisors of the semigroup module $M[S]$. Particularly we show that if $M$ is an $R$-module and $S$ a commutative, cancellative and torsion-free monoid, then the $R[S]$-module $M[S]$ has few zero-divisors of size $n$ if and only if the $R$-module $M$ has few zero-divisors of size $n$ and Property (A).
\end{abstract}

\maketitle

\section{Introduction}

Let $S$ be a commutative semigroup and $M$ be an $R$-module. One can define the semigroup module $M[S]$ as an $R[S]$-module constructed from the semigroup $S$ and the $R$-module $M$ similar to the standard definition of semigroup rings. Obviously similar to semigroup rings, the zero-divisors of the semigroup module $M[S]$ are interesting to investigate (\cite[p. 82]{G1} and \cite{J}).

$ $

We write each element of $g \in M[S]$ as ``polynomials'' $g=m_1X^{s_1}+m_2X^{s_2}+\cdots+m_nX^{s_n}$, where $m_1,\ldots,m_n \in M$ and $s_1,\ldots,s_n$ are distinct elements of $S$ and this representation of $g$ is called the canonical form of $g$. For $g = m_1X^{s_1}+m_2X^{s_2}+\cdots+m_nX^{s_n}$, we define the content of $g$ to be the $R$-submodule of $M$ generated by the coefficients of $g$.

$ $

Northcott gave a nice generalization of \textit{Dedekind-Mertens Lemma} as follows: if $S$ is a commutative, cancellative and torsion-free monoid and $M$ is an $R$-module, then for all $f \in R[S]$ and $g \in M[S]$, there exists a natural number $k$ such that $c(f)^k c(g) = c(f)^{k-1} c(fg)$ (\cite{N2}). Dedekind-Mertens Lemma has different versions with various applications (\cite{AG}, \cite{AK}, \cite{BG}, \cite{GGP}, \cite{HH}, \cite{N1}, \cite{NY}, \cite{OR}, and \cite{R}). One of its interesting consequences is McCoy's Theorem on zero-divisors (\cite[p. 96]{G1} and \cite{M}): If $M$ is a nonzero $R$-module and $S$ is a commutative, cancellative and torsion-free monoid, then for all $f \in R[S]$ and $g \in M[S]-\lbrace0\rbrace$, if $fg = 0$, then there exists an $m \in M-\lbrace 0 \rbrace$ such that $f\cdot m = 0$.

$ $

An $R$-module $M$ is said to have \textit{few zero-divisors of size} $n$, if $Z_R(M)$ is a finite union of $n$ prime ideals $\textbf{p}_1,\ldots,\textbf{p}_n$ of $R$ such that $\textbf{p}_i \nsubseteq \textbf{p}_j$ for all $i \neq j$. Also note that an $R$-module $M$ has \textit{Property (A)}, if each finitely generated ideal $I \subseteq Z_R(M)$ has a nonzero annihilator in $M$. We use McCoy's Theorem to prove that if $M$ is an $R$-module and $S$ a commutative, cancellative and torsion-free monoid, then the $R[S]$-module $M[S]$ has few zero-divisors of size $n$, if and only if the $R$-module $M$ has few zero-divisors of size $n$ and Property (A).

In this paper all rings are commutative with identity and all modules are unital\footnote{2000 Mathematics Subject Classification: 13A15, 13B25, 13C05, 20M14}. Unless otherwise stated, our notation and terminology will follow as closely as possible that of Gilmer \cite{G1}.

\section{Zero-Divisors of Semigroup Modules}

Let us recall that if $R$ is a ring and $f=a_0 + a_1X + \cdots + a_nX^n$ is a polynomial on the ring $R$, then content of $f$ is defined as the $R$-ideal, generated by the coefficients of $f$, i.e. $c(f)= (a_0, a_1, \ldots, a_n)$. The content of an element of a semigroup module is a natural generalization of the content of a polynomial as follows:

\begin{definition}
 Let $M$ be an $R$-module and $S$ be a commutative semigroup. Let $g \in M[S]$ and put $g = m_1X^{s_1}+m_2X^{s_2}+\cdots+m_nX^{s_n}$, where $m_1,\ldots,m_n \in M$ and $s_1,\ldots,s_n \in S$. We define the content of $g$ to be the $R$-submodule of $M$ generated by the coefficients of $g$, i.e. $c(g) = (m_1,\ldots,m_n)$.
\end{definition}

\begin{theorem}
 Let $S$ be a commutative monoid and $M$ be a nonzero $R$-module. Then the following statements are equivalent:
\begin{enumerate}
 \item $S$ is a cancellative and torsion-free monoid.
 \item For all $f \in R[S]$ and $g \in M[S]$, there is a natural number $k$ such that $c(f)^k c(g) = c(f)^{k-1} c(fg)$.
 \item (\textit{McCoy's Property}) For all $f \in R[S]$ and $g \in M[S]-\lbrace0\rbrace$, if $fg = 0$, then there exists an $m \in M-\lbrace 0 \rbrace$ such that $f\cdot m = 0$.
\item For all $f \in R[S]$, $\Ann_M(c(f)) =0$ if and only if $f \notin Z_{R[S]}(M[S])$.

\end{enumerate}

\end{theorem}

\begin{proof}
$(1) \rightarrow (2)$ has been proved in \cite{N2}.

$ $

For $(2) \rightarrow (3)$, assume that $f \in R[S]$ and $g \in M[S]-\lbrace0\rbrace$, such that $fg = 0$. So there exists a natural number $k$ such that $c(f)^k c(g) = c(f)^{k-1} c(fg)=(0)$. Take $t$ the smallest natural number such that $c(f)^t c(g)=(0)$ and choose $m$ a nonzero element of $c(f)^{t-1}c(g)$. It is easy to check that $f\cdot m=0$.

$ $

For $(3) \rightarrow (1)$, we prove that if $S$ is not cancellative or not torsion-free then (1) cannot hold. For the moment, suppose that $S$ is not cancellative, so there exist $s,t,u \in S$ such that $s+t = s+u$ while $t \not= u$. Put $f = X^s$ and $g = (qX^{t}-qX^u)$, where $q$ is a nonzero element of $M$. Then obviously $fg = 0$, while $f\cdot m\neq 0$ for all $m \in M-\lbrace 0 \rbrace$. Finally suppose that $S$ is cancellative but not torsion-free. Let $s,t \in S$ be such that $s \not=t$, while $ns = nt$ for some natural $n$. Choose the natural number $k$ minimal so that $ks = kt$. Then we have $0 = qX^{ks}-qX^{kt} = (\sum_{i=0}^{k-1} X^{(k-i-1)s+it}) (qX^s-qX^t)$, where $q$ is a nonzero element of $M$.

Since $S$ is cancellative, the choice of $k$ implies that $ (k-i_1-1)s+i_{1}t \not= (k-i_2-1)s+i_{2}t $ for $0 \leq i_1 < i_2 \leq k-1 $.
Therefore $\sum_{i=0}^{k-1} X^{(k-i-1)s+it} \not= 0$, and this completes the proof. $(3) \leftrightarrow (4)$ is obvious.
\end{proof}

\begin{corollary}
 Let $M$ be an $R$-module and $S$ be a commutative, cancellative and torsion-free monoid. Then the following statements hold:
\begin{enumerate}
 \item $R$ is a  domain if and only if $R[S]$ is a domain.
 \item If $\textbf{p}$ is a prime ideal of $R$, then $\textbf{p}[S]$ is a prime ideal of $R[S]$.
 \item If $\textbf{p}$ is in $\Ass_R(M)$, then $\textbf{p}[S]$ is in $\Ass_{R[S]}(M[S])$.
\end{enumerate}

\end{corollary}

\begin{definition}
 Let $M$ be an $R$-module and $P$ be a proper $R$-submodule of $M$. $P$ is said to be a \textit{prime submodule} (\textit{primary submodule}) of $M$, if $rx \in P$ implies $x \in P$ or $rM \subseteq P$ (there exists a natural number $n$ such that $r^nM \subseteq P$), for each $r \in R$ and $x \in M$.
\end{definition}

\begin{corollary}
 Let $M$ be an $R$-module and $S$ be a commutative, cancellative and torsion-free monoid. Then the following statements hold:
\begin{enumerate}
 \item $(0)$ is a  prime (primary) submodule of $M$ if and only if $(0)$ is a prime (primary) submodule of $M[S]$.
 \item If $P$ is a prime (primary) submodule of $M$, then $P[S]$ is a prime (primary) submodule of $M[S]$.
\end{enumerate}

\end{corollary}

In \cite{D}, it has been defined that a ring $R$ has \textit{few zero-divisors}, if $Z(R)$ is a finite union of prime ideals. We give the following definition and prove some interesting results about zero-divisors of semigroup modules. Modules having (very) few zero-divisors, introduced in \cite{N1}, have also some interesting homological properties \cite{NP}.

 \begin{definition}
  An $R$-module $M$ has \textit{very few zero-divisors}, if $Z_R(M)$ is a finite union of prime ideals in $\Ass_R(M)$.
 \end{definition}

\begin{remark}
 \textit{Examples of modules having very few zero-divisors}. If $R$ is a Noetherian ring and $M$ is an $R$-module such that $\Ass_R(M)$ is finite, then obviously $M$ has very few zero-divisors. For example $\Ass_R(M)$ is finite if $M$ is a finitely generated $R$-module \cite[p. 55]{K}. Also if $R$ is a Noetherian quasi-local ring and $M$ is a balanced big Cohen-Macaulay $R$-module, then $\Ass_R(M)$ is finite \cite[Proposition 8.5.5, p. 344]{BH}.
\end{remark}

\begin{remark}
Let $R$ be a ring and consider the following three conditions on $R$:

\begin{enumerate}
 \item $R$ is a Noetherian ring.
 \item $R$ has very few zero-divisors.
 \item $R$ has few zero-divisors.
\end{enumerate}

Then, $(1) \rightarrow (2) \rightarrow (3)$ and none of the implications are reversible.
\end{remark}

\begin{proof}
For  $(1) \rightarrow(2) $ use \cite[p. 55]{K}. It is obvious that $(2) \rightarrow (3)$.

Suppose $k$ is a field, $A=k[X_1, X_2, X_3,\ldots,X_n,\ldots]$ and $\textbf{m} =(X_1, X_2, X_3,\ldots, X_n,\ldots)$ and at last $\textbf{a}=(X_1^2, X_2^2, X_3^2,\ldots, X_n^2,\ldots)$. Since $A$ is a domain, $A$ has very few zero-divisors while it is not a Noetherian ring. Also consider the ring $R=A/\textbf{a}$. It is easy to check that $R$ is a quasi-local ring with the only prime ideal $\textbf{m}/\textbf{a}$ and $Z(R)=\textbf{m}/\textbf{a}$ and finally $\textbf{m}/\textbf{a}\notin \Ass_R(R)$. Note that $\Ass_R(R)=\emptyset$ \cite{N1}.
\end{proof}

\begin{theorem}
 Let $M$ be an $R$-module and $S$ a commutative, cancellative and torsion-free monoid. Then the $R[S]$-module $M[S]$ has very few zero-divisors, if and only if the $R$-module $M$ has very few zero-divisors.
\end{theorem}

\begin{proof}
 $(\leftarrow)$: Let $Z_R(M) = {\textbf{p}_1}\cup {\textbf{p}_2}\cup \cdots \cup {\textbf{p}_n}$, where ${\textbf{p}_i} \in \Ass_R(M)$ for all $1 \leq i \leq n$. We will show that $Z_{R[S]}(M[S]) = {\textbf{p}_1}[S]\cup {\textbf{p}_2}[S]\cup \cdots \cup {\textbf{p}_n}[S]$. Let $f \in Z_{R[S]}(M[S])$, so there exists an $m \in M- \lbrace 0 \rbrace $ such that $f\cdot m = 0$ and so $c(f)\cdot m = (0)$. Therefore $c(f) \subseteq Z_R(M)$ and this means that $c(f) \subseteq {\textbf{p}_1}\cup {\textbf{p}_2}\cup \cdots \cup {\textbf{p}_n}$ and according to the Prime Avoidance Theorem, we have $c(f) \subseteq {\textbf{p}_i}$, for some $1 \leq i \leq n$ and therefore $f \in {\textbf{p}_i}[S]$. Now let $f \in {\textbf{p}_1}[S]\cup {\textbf{p}_2}[S]\cup \cdots \cup {\textbf{p}_n}[S]$, so there exists an $i$ such that $f \in {\textbf{p}_i}[S]$, so $c(f) \subseteq {\textbf{p}_i}$ and $c(f)$ has a nonzero annihilator in $M$ and this means that $f$ is a zero-divisor of $M[S]$. Note that by Corollary 3, ${\textbf{p}_i}[S] \in \Ass_{R[S]}(M[S])$ for all $1 \leq i \leq n$.

$(\rightarrow)$: Let $Z_{R[S]}(M[S])= \cup _{i=1}^n Q_i$, where $Q_i \in \Ass_{R[S]}(M[S])$ for all $1\leq i \leq n$. Therefore $Z_R(M) = \cup _{i=1}^n (Q_i \cap R)$. Without loss of generality, we can assume that $Q_i \cap R \nsubseteq Q_j \cap R$ for all $i \neq j$. Now we prove that $Q_i \cap R\in \Ass_R(M)$ for all $1 \leq i \leq n$. Consider $g\in M[S]$ such that $Q_i = \Ann (g)$ and $g = m_1X^{s_1}+m_2X^{s_2}+\cdots+m_nX^{s_n}$, where $m_1,\ldots,m_n \in M$ and $s_1,\ldots,s_n \in S$. It is easy to see that $Q_i \cap R = \Ann (c(g)) \subseteq \Ann(m_1) \subseteq Z_R(M)$ and by the Prime Avoidance Theorem, $Q_1 \cap R =\Ann(m_1)$.
\end{proof}

In \cite{HK}, it has been defined that a ring $R$ has \textit{Property (A)}, if each finitely generated ideal $I \subseteq Z(R)$ has a nonzero annihilator. We give the following definition:

\begin{definition}
 An $R$-module $M$ has \textit{Property (A)}, if each finitely generated ideal $I \subseteq Z_R(M)$ has a nonzero annihilator in $M$.
\end{definition}

\begin{remark}
 If the $R$-module $M$ has very few zero-divisors, then $M$ has Property (A).
\end{remark}

\begin{theorem}
 Let $S$ be a commutative, cancellative and torsion-free monoid and $M$ be an $R$-module. The following statements are equivalent:

\begin{enumerate}
 \item The $R$-module $M$ has Property (A).
 \item For all $f \in R[S]$, $f$ is $M[S]$-regular if and only if $c(f)$ is $M$-regular.
\end{enumerate}

\begin{proof}
 $(1) \rightarrow (2)$: Let the $R$-module $M$ have Property (A). If $f \in R[S]$ is $M[S]$-regular, then $f \cdot m \not= 0$ for all nonzero $m \in M$ and so $c(f)\cdot m \not= (0)$ for all nonzero $m \in M$ and according to the definition of Property (A), $c(f) \not\subseteq Z_R(M)$. This means that $c(f)$ is $M$-regular. Now let $c(f)$ be $M$-regular, so $c(f) \not\subseteq Z_R(M)$ and this means that $c(f)\cdot m \not= (0)$ for all nonzero $m \in M$ and hence $f\cdot m \not= 0$ for all nonzero $m \in M$. Since $S$ is a commutative, cancellative and torsion-free monoid, $f$ is not a zero-divisor of $M[S]$, i.e. $f$ is $M[S]$-regular.

$(2) \rightarrow (1)$: Let $I$ be a finitely generated ideal of $R$ such that $I \subseteq Z_R(M)$. Then there exists an $f \in R[S]$ such that $c(f) = I$. But $c(f)$ is not $M$-regular, therefore according to our assumption, $f$ is not $M[S]$-regular. Therefore there exists a nonzero $m \in M$ such that $f\cdot m = 0$ and this means that $I\cdot m = (0)$, i.e. $I$ has a nonzero annihilator in $M$.
\end{proof}

\end{theorem}

Let, for the moment, $M$ be an $R$-module such that the set $Z_R(M)$ of zero-divisors of $M$ is a finite union of prime ideals. One can consider $Z_R(M)= \cup _{i=1}^n \textbf{p}_i$ such that $\textbf{p}_i \nsubseteq \cup _{j=1 \wedge j \neq i}^n \textbf{p}_j$ for all $ 1\leq i \leq n$. Obviously we have $\textbf{p}_i \nsubseteq \textbf{p}_j$ for all $i \neq j$. Also, it is easy to check that, if $Z_R(M)= \cup _{i=1}^n \textbf{p}_i$ and $Z_R(M)= \cup _{k=1}^m \textbf{q}_k$ such that $\textbf{p}_i \nsubseteq \textbf{p}_j$ for all $i \neq j$ and $\textbf{q}_k \nsubseteq \textbf{q}_l$ for all $k \neq l$, then $m=n$ and $\{\textbf{p}_1,\ldots,\textbf{p}_n\}=\{\textbf{q}_1,\ldots,\textbf{q}_n\}$, i.e. these prime ideals are uniquely determined (Use the Prime Avoidance Theorem). This is the base for the following definition:

\begin{definition}
 An $R$-module $M$ is said to have \textit{few zero-divisors of size} $n$, if $Z_R(M)$ is a finite union of $n$ prime ideals $\textbf{p}_1,\ldots,\textbf{p}_n$ of $R$ such that $\textbf{p}_i \nsubseteq \textbf{p}_j$ for all $i \neq j$.
\end{definition}

\begin{theorem}
 Let $M$ be an $R$-module and $S$ a commutative, cancellative and torsion-free monoid. Then the $R[S]$-module $M[S]$ has few zero-divisors of size $n$, if and only if the $R$-module $M$ has few zero-divisors of size $n$ and Property (A).
\end{theorem}

\begin{proof}

$(\leftarrow)$: By considering the $R$-module $M$ having Property (A), similar to the proof of Theorem 9, we have if $Z_R(M)= \cup _{i=1}^n \textbf{p}_i$, then $Z_{R[S]}(M[S])= \cup _{i=1}^n \textbf{p}_i[S]$. Also it is obvious that $\textbf{p}_i[S] \subseteq \textbf{p}_j[S]$ if and only if $\textbf{p}_i \subseteq \textbf{p}_j$ for all $1 \leq i,j \leq n$. These two imply that the $R[S]$-module $M[S]$ has few zero-divisors of size $n$.

$(\rightarrow)$: Note that $Z_R(M) \subseteq Z_{R[S]}(M[S])$. It is easy to check that if $Z_{R[S]}(M[S])= \cup _{i=1}^n Q_i$, where $Q_i$ are prime ideals of $R[S]$ for all $1\leq i \leq n$, then $Z_R(M) = \cup _{i=1}^n (Q_i \cap R)$. Now we prove that the $R$-module $M$ has Property (A). Let $I \subseteq Z_R(M)$ be a finite ideal of $R$. Choose $f\in R[S]$ such that $I=c(f)$. So $c(f) \subseteq Z_R(M)$ and obviously $f \in Z_{R[S]}(M[S])$ and according to McCoy's property, there exists a nonzero $m\in M$ such that $f\cdot m=0$. This means that $I\cdot m=0$ and $I$ has a nonzero annihilator in $M$. Consider that by a similar discussion in $(\leftarrow)$, the $R$-module $M$ has few zero-divisors obviously not less than size $n$ and this completes the proof.
\end{proof}

An $R$-module $M$ is said to be \textit{primal}, if $Z_R(M)$ is an ideal of $R$ \cite{D}. It is easy to check that if $Z_R(M)$ is an ideal of $R$, then it is a prime ideal and therefore the $R$-module $M$ is primal if and only if $M$ has few zero-divisors of size one.

\begin{corollary}
 Let $M$ be an $R$-module and $S$ a commutative, cancellative and torsion-free monoid. Then the $R[S]$-module $M[S]$ is primal, if and only if the $R$-module $M$ is primal and has Property (A).
\end{corollary}

\section{Acknowledgment} The author wishes to thank Prof. Winfried Bruns and the referee for their useful advice and Dr. Neil Epstein for his editorial comments.

\end{document}